\documentclass[12pt]{amsart}

\usepackage{amssymb}
\usepackage{amsfonts}
\usepackage{amsmath}
\usepackage{mathrsfs}
\usepackage[curve,matrix,arrow,dvips,ps]{xy}

\def\C{\mathrm{C}}
\def\CA{{\mathscr{A}}}
\def\CF{{\mathcal{F}}}
\def\CK{{\mathscr{K}}}
\def\CX{{\mathfrak{X}}}
\def\CY{{\mathfrak{Y}}}
\def\Hom{\mathrm{Hom}}
\def\Br{\mathrm{Br}}
\def\Id{\mathrm{Id}}
\def\Z{\mathrm{Z}}

\newtheorem{thm}[equation]{Theorem}

\theoremstyle{definition}
\newtheorem{defi}[equation]{Definition}
\newtheorem{exm}[equation]{Example}

\begin{document}
\title{Commuting categories for blocks and fusion systems}                                                    
\author{Adam Glesser}
\address{Suffolk University, 8 Ashburton Place, Boston, MA 02108, USA}
\email{aglesser@suffolk.edu}
\author{Markus Linckelmann}
\address{Institute of Mathematics, University of Aberdeen, 
Aberdeen AB24 3UE, UK}
\email{m.linckelmann@abdn.ac.uk}

\date{\today}

\begin{abstract}
We extend the notion of a commuting poset for a finite group
to $p$-blocks and fusion systems, and we generalize a result,
due originally to Alperin and proved independently by Aschbacher 
and Segev, to commuting graphs of blocks, with a very short proof 
based on the $G$-equivariant version, due to Th\'evenaz and Webb, 
of a result of Quillen.
\end{abstract}

\thanks{The authors completed much of this work while enjoying the hospitality of MSRI in 2008.}

\maketitle

Let $k$ be a field of prime characteristic $p$. A {\it block of
a finite group} $G$ is a primitive idempotent $b$ in $\Z(kG)$. 
A $b$-{\it Brauer pair} is a pair $(Q,e)$ consisting of a
$p$-subgroup $Q$ of $G$ and a block $e$ of $\C_G(Q)$ satisfying
$\Br_Q(b)e \neq 0$, where $\Br_Q : (kG)^Q\to$ $k\C_G(Q)$ is the
Brauer homomorphism; the set of $b$-Brauer pairs is a $G$-poset with
respect to the conjugation action of $G$ (see \cite{Thev} for more details 
and background  material on block theory). We denote by $\CA(b)$ 
the $G$-poset containing all $b$-Brauer pairs $(Q,e)$ 
such that $Q$ is nontrivial and elementary abelian.

Two subgroups $R$, $R'$ of $G$ are said to {\it commute} if they
commute elementwise; that is, if $[R,R']=$ $1$.  For any nonempty 
set $\kappa$ of pairwise commuting subgroups of $G$ we denote by 
$\Pi\kappa$ the product in $G$ of all subgroups belonging to 
$\kappa$; this is clearly a subgroup of $G$. If all elements of 
$\kappa$ are $p$-subgroups (respectively, abelian subgroups) of $G$, 
then $\Pi\kappa$ is a $p$-subgroup (respectively, abelian subgroup) 
of $G$. For any abelian subgroup $Q$ of $G$ we denote by $c(Q)$ the set
of subgroups of order $p$ of $Q$.

\begin{defi}
Let $G$ be a finite group and $b$ a block of $G$. The {\it commuting 
poset of $b$} is the $G$-poset $\CK(b)$ whose elements are pairs 
$(\kappa, e)$, where $\kappa$ is a nonempty set of pairwise commuting 
subgroups of order $p$ of $G$ and where $e$ is a block of 
$\C_G(\Pi\kappa)$ such that $(\Pi\kappa,e)$ is a $b$-Brauer pair, with 
partial order given by 
\[
(\lambda,f)\leq (\kappa,e) \text{, if } 
	\begin{cases} 
	\lambda\subseteq \kappa, \text{and }\\ 
	(\Pi\lambda,f)\leq (\Pi\kappa,e)
	\end{cases}
\]
for $(\kappa,e)$, $(\lambda,f) \in$ $\CK(b)$.
\end{defi}

If $b$ is the principal block of $G$ then $\CK(b)$ is the clique complex
$\CK_p(G)$ of the commuting graph $\Lambda_p(G)$, where the notation is 
as in \cite{As}. For nonprincipal blocks, however, $\CK(b)$ need not be 
the clique complex of a graph (e.g., see Example \ref{nonclique}).

Given a $G$-poset $\CX$ we denote by $\Delta \CX$ the $G$-simplicial 
complex whose set of $n$-simplices consists of 
all chains of $n$ proper inclusions in $X$, 
where $n \geq 0$. For any simplicial complex $\CY$, we denote the 
geometric realization of $\CY$ by $|\CY|$. Two $G$-spaces $X$ and $Y$ are called {\it $G$-homotopically equivalent} if there are $G$-equivariant maps $f : X \to Y$, $g : Y \to X$ 
and $G$-equivariant homotopies $h : I \times X \to X$, 
$h' : I \times Y \to Y$ such that $h(0,-) = \Id_A$, $h(1,-) = f$, 
$h'(0,-) = \Id_Y$, and $h'(1,-) = g$, where the unit interval $I = [0,1]$ 
is viewed as a $G$-space with the trivial $G$-action. Two $G$-posets $\CX$ and 
$\CY$ are called {\it $G$-homotopically equivalent} if the $G$-spaces 
$|\Delta \CX|$ and $|\Delta \CY|$ are $G$-homotopically equivalent. 
By the $G$-equivariant version \cite[(1.1)]{ThWe} of \cite[1.3]{Qu78},
in order to show that $\CX$ and $\CY$ are $G$-homotopically equivalent, it
suffices to find $G$-equivariant functors
$\Phi : \CX\to$ $\CY$ and $\Psi : \CY\to$ $\CX$ such that there 
is a natural transformation between $\Id_\CX$ and $\Psi\circ\Phi$ 
(in either direction) and a natural transformation between $\Id_\CY$ 
and $\Phi\circ\Psi$.

\begin{thm} \label{theorem1}
Let $b$ be a block of a finite group $G$. The maps:
\\

\mbox{
$\Phi : \begin{cases}
       & \CA(b) \rightarrow \CK(b) \\
       & (Q,e) \mapsto (c(Q), e)
\end{cases}$ 
\qquad and \qquad 
$\Psi : \begin{cases}
       & \CK(b)\rightarrow \CA(b) \\
       & (\kappa,e) \mapsto (\Pi\kappa,e)
\end{cases}$
}
\\

are inverse $G$-homotopy equivalences.
\end{thm}

\begin{proof}
The maps $\Phi$, $\Psi$ are obviously order preserving and
$G$-equivariant. We have $\Psi\circ\Phi=\Id_{\CA(b)}$.
There is a natural transformation $\Id_{\CK(b)}\rightarrow$
$\Phi\circ\Psi$ given by $(\kappa,e)\leq (c(\Pi\kappa),e)$,
which shows that $\Psi$ is a $G$-homotopy inverse of $\Phi$.
\end{proof}

Applied to principal blocks, this theorem yields, in 
particular, a proof of the fact, due independently 
to Alperin \cite[Theorem 3]{Al2} and to Aschbacher and Segev 
\cite[9.7]{AsSe}, that $\CK_p(G)$ and $\CA_p(G)$ have the same 
homotopy type (see also \cite[5.2]{As}).  The $G$-orbit space of 
$\CK(b)$ admits a generalization to fusion systems and, in fact, 
to arbitrary categories on finite $p$-groups (cf. \cite[2.1]{Linfus}).

\begin{defi}
Let $\CF$ be a category on a finite $p$-group $P$. The  {\it commuting 
category} of $\CF$ is the category $\CK(\CF)$ whose objects are the 
nonempty sets of pairwise commuting subgroups of $P$ of order $p$,
and for objects $\kappa, \lambda \in \CK(\CF)$, 
\[
\Hom_{\CK(\CF)}(\kappa, \lambda) = 
   \{\psi \in \Hom_\CF(\Pi \kappa, \Pi \lambda) \mid 
   \text{if $Q \in \kappa$, then $\psi(Q) \in \lambda$}.\}
\]
The composition of morphisms in $\CK(\CF)$ is induced by the usual
composition of group homomorphisms. We denote by $[\CK(\CF)]$ the 
poset consisting of the isomorphism classes $[\kappa]$ of objects 
$\kappa$ of $\CK(\CF)$ with partial order given by
\[
[\kappa] \leq [\lambda] \text{, if } \Hom_{\CK(\CF)}(\kappa, \lambda) \neq 
\varnothing
\]
for $\kappa$, $\lambda\in \CK(\CF)$. 
\end{defi}

Clearly $\CK(\CF)$ is an $EI$-category. 
As a consequence of results in \cite{AlBr}, any choice of a maximal 
$b$-Brauer pair $(P,e)$ of a block $b$ of a finite group $G$ determines 
a category $\CF_{(P,e)}(G,b)$ on $P$ that, if $k$ is large enough, is a 
saturated fusion system (see e.g., \cite[\S 3.3]{Kessar} for details 
and further references).

\begin{thm}\label{theorem2}
Let $b$ be a block of a finite group $G$, let $(P,e_P)$ be a 
maximal $b$-Brauer pair and let $\CF =$ $\CF_{(P,e_P)}(G, b)$. 
We have an isomorphism of posets
$$[\CK(\CF)] \cong \CK(b)/G$$
mapping the isomorphism class of an object $\kappa\in\CK(\CF)$ to 
the $G$-conjugacy class of the unique Brauer pair $(\Pi\kappa,e)$ 
contained in $(P,e_P)$.
\end{thm}

\begin{proof}
For $(\kappa, e) \in$ $\CK(b)$, let $[(\kappa, e)]$ denote its 
$G$-conjugacy class. For elements $(\kappa, e), (\lambda, f) \in$ $\CK(b)$, 
one has $[(\kappa, e)] =$ $[(\lambda, f)]$ if and only if there exists 
$g \in G$ such that $\kappa^g = $ $\lambda$ and $e^g =$ $f$. Define 
a poset map $\eta: \CK(b)/G  \to $ $[\CK(\CF)]$ by setting 
$\eta([(\kappa, e)]) =$ $[\kappa^g]$,
where $g \in G$ such that $(\Pi \kappa, e)^g \leq$ $(P, e_P)$.
One verifies that this map is the inverse of the given map in
the statement.
\end{proof}

\begin{exm} \label{nonclique}
The following example was communicated to the authors by R. Kessar.
Suppose $p=2$. Set $G=$ $S_n$, where $n\geq 6$ is an integer such 
that $kG$ has a block $b$ with a dihedral defect group $P\cong$ $D_8$ 
of order $8$. 
By results in \cite{PuigNakayama}, $b$ is of principal type; that
is, for any $2$-subgroup $Q$ of $G$ either $\Br_Q(b)=$ $0$ or
$\Br_Q(b)$ is a block of $k\C_G(Q)$. Moreover, $P$ may be chosen as a 
Sylow $2$-subgroup of $S_4$, canonically embedded into $G$ and such that 
$P$ contains the involutions $x= (1\;2)$, $y= (3\;4)$. Setting $z = (5\;6)$, 
we have $x,z \in P^{(3\;5)(4\;6)}$ and $y,z \in P^{(1\;5)(2\;6)}$.
Since $b$ is of principal type, there are unique
blocks $e_x$, $e_y$, $e_z$ of $k\C_G(x)$, $k\C_G(y)$, $k\C_G(z)$,
respectively, and unique blocks $e_{xy}, e_{xz}, e_{yz}$ of $k\C_G(\langle x,y \rangle)$, $k\C_G(\langle x,z \rangle)$, $k\C_G(\langle y,z \rangle)$, respectively, giving the following inclusions of $b$-Brauer pairs:

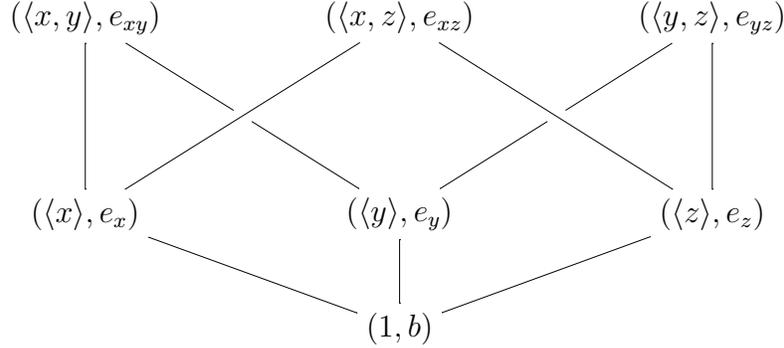
\[
\xymatrix{
(\langle x,y \rangle, e_{xy}) \ar@{-}[dd] \ar@{-} '[dr][ddrr] & & (\langle x,z \rangle, e_{xz}) \ar@{-}[ddll] \ar@{-}[ddrr] & & (\langle y,z \rangle, e_{yz}) \ar@{-}'[dl][ddll] \ar@{-}[dd] \\ & & & & \\ (\langle x \rangle,e_x) \ar@{-}[drr] & & (\langle y \rangle,e_y) \ar@{-}[d] & & (\langle z\rangle ,e_z) \ar@{-}[dll] \\ & & (1,b) & & 
}
\]

Suppose that $\Gamma$ is a graph whose clique complex is
$\CK(b)$. The $b$-Brauer pairs
$(\langle x\rangle, e_x)$, 
$(\langle y\rangle, e_y)$, and $(\langle z\rangle, e_z)$ are
minimal in the poset $\CK(b)$ and are pairwise contained in a common $b$-Brauer pair, implying that the graph $\Gamma$ has
a clique of the form:

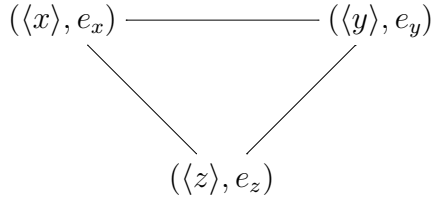
\[
\xymatrix@!@-1pc{
(\langle x\rangle,e_x) \ar@{-}[rr] \ar@{-}[dr] & & (\langle y\rangle,e_y) \ar@{-}[dl] \\ & (\langle z\rangle,e_z) &
}
\]

However, the corresponding clique is not an element of the poset $\CK(b)$ because the group $\langle x,y,z\rangle$ is not contained in a
defect group of $b$. This contradiction shows that there
is no graph whose clique complex yields $\CK(b)$ and explains 
why we have refrained from defining a commuting graph of $b$ 
in this way.
\end{exm}

\end{document}